\documentclass[10pt]{article}

\usepackage{amsmath,amssymb,amscd,amsthm,makeidx,txfonts,graphicx,url,bm}
\usepackage{a4wide,pstricks,xypic}

\numberwithin{equation}{subsection}

\let\oldmarginpar\marginpar
\renewcommand\marginpar[1]{\-\oldmarginpar[\raggedleft\footnotesize #1]
{\raggedright\footnotesize #1}}

\makeindex

\xyoption{all}
\input{xypic}

\newtheorem{theorem}{Theorem}[subsection]
\newtheorem{proposition}[theorem]{Proposition}
\newtheorem{corollary}[theorem]{Corollary}

\newtheorem{lemma}[theorem]{Lemma}

\theoremstyle{remark}
\newtheorem{remark}[theorem]{Remark}

\theoremstyle{definition}

\def\beq{\begin{eqnarray}}
\def\eeq{\end{eqnarray}}
\def\bes{\begin{eqnarray*}}
\def\ees{\end{eqnarray*}}

\DeclareMathOperator{\Hua}{Hua}
\DeclareMathOperator{\Li}{Li}
\DeclareMathOperator{\Stab}{Stab}

\DeclareMathOperator{\diag}{diag}

\def\C{\mathbb{C}}
\def\M{{\mathcal{M}}}

\def\z{\mathbf{z}}

\def\calH{\mathcal{H}}

\def\F{\mathbb{F}}
\def\Q{\mathbb{Q}}

\def\Z{\mathbb{Z}}

\newcommand{\nc}{\newcommand}

\nc{\op}[1]{\mathop{\mathchoice{\mbox{\rm #1}}{\mbox{\rm #1}}
{\mbox{\rm \scriptsize #1}}{\mbox{\rm \tiny #1}}}\nolimits}
\nc{\al}{\alpha}

\nc{\ep}{\varepsilon} \nc{\ga}{\gamma} \nc{\Ga}{\Gamma}
\nc{\la}{\lambda} \nc{\La}{\Lambda} \nc{\si}{\sigma}
\nc{\Sig}{{\Gamma}} \nc{\Om}{\Omega} \nc{\om}{\omega}

\nc{\SL}{{\rm SL}}
\nc{\GL}{{\rm GL}}
\nc{\PGL}{{\rm PGL}}
\nc{\G}{{\rm G}}

\nc{\Frob}{\op{ Frob}}
\nc{\Irr}{\op{Irr}}

\nc{\cpt}{{\op{cpt}}} \nc{\Dol}{{\op{Dol}}} \nc{\DR}{{\op{DR}}}
\nc{\B}{{\op{B}}} \nc{\Triv}{\op{Triv}} \nc{\Hod}{{\op{Hod}}}
\nc{\Log}{{\op{Log}}} \nc{\Exp}{{\op{Exp}}} \nc{\Est}{E_{\op{st}}}
\nc{\Hst}{H_{\op{st}}} \nc{\Left}[1]{\hbox{$\left#1\vbox to
   10.5pt{}\right.\nulldelimiterspace=0pt \mathsurround=0pt$}}
\nc{\Right}[1]{\hbox{$\left.\vbox to
   10.5pt{}\right#1\nulldelimiterspace=0pt \mathsurround=0pt$}}
\nc{\LEFT}[1]{\hbox{$\left#1\vbox to
   15.5pt{}\right.\nulldelimiterspace=0pt \mathsurround=0pt$}}
\nc{\RIGHT}[1]{\hbox{$\left.\vbox to
   15.5pt{}\right#1\nulldelimiterspace=0pt \mathsurround=0pt$}}

\nc{\bee}{{\bf E}} \nc{\bphi}{{\bf \Phi}}

\begin{document}

\title{A refinement of the $A$-polynomial of quivers} 
\author{Fernando  Rodriguez Villegas}

\maketitle

\section{Introduction}
For a finite quiver $Q=(Q_0,Q_1)$ Kac proved that the number of
absolutely irreducible representations of $Q$ of dimension vector
$n\in \Z_{\geq 0}^{Q_0}$ over $\F_q$ is given by a polynomial
$A_n(q)$. He conjectured that $A_n$ has non-negative integer
coefficients.

In this preliminary note we propose an refinement of this
polynomial. For simplicity we will concentrate on $S_g$, the quiver
consisting of one vertex and $g$ loops, though we expect many of the
results to extend to a general quiver. We only give a brief sketch of
proofs, fuller details will appear elsewhere.

We define  a priori rational functions $A_\lambda(q)$ indexed by
partitions $\lambda$ which give a decomposition
$$
A_n(q)=\sum_{|\lambda|=n}A_\lambda(q)
$$
of the $A$-polynomial of $S_g$.  (We drop $g$ from the notation if
there is no risk of confusion.)

Computations suggest that for $g>0$, which we assume from now on,
$A_\lambda(q)$ is in fact a polynomial in $q$ with non-negative
integer coefficients.  For example, for $g=2$ and $n=3$ we obtain
$$
A_{(1,1,1)}(q)=q^{10} + q^8 + q^7, \quad
A_{(2,1)}=q^6 + q^5, \quad
A_{(3)}=q^4
$$
with sum
$$
A_3(q)=q^{10} + q^8 + q^7 + q^6 + q^5 + q^4
$$
It would be quite interesting to understand this decomposition
directly in terms of absolutely indecomposable representations of the
quiver.

We sketch below a proof of the following formula for $A_\lambda(1)$
(the case $\lambda=(1^n)$ was previously proved by
Reineke~\cite{Reineke09} by different methods). By the conjectures of
\cite{HRV} the number $A_n(1)=\sum_{|\lambda|=n}A_\lambda(1)$ should
equal the dimension of the middle dimensional cohomology group of the
character variety $\M_n$ studied there. A refined version of this
conjecture states that $A_\lambda(1)$ is the number of connected
components of type $\lambda$ of a natural $\C^\times$ action on the
moduli space of Higgs bundles, which is diffeomorphic to $\M_n$. A
proof of this conjecture for $\lambda=(1^n)$ was recently given by
Reineke \cite[Theorem 7.1]{Reineke11}. The refined conjecture
originates in \cite[Remark 4.4.6]{HRV} and was in fact the motivation
to construct the truncated polynomials $A_\lambda(q)$ studied in this
note.

\begin{theorem}
\label{main}
For any non-zero partition $\lambda$ we have
\begin{equation}
\label{main-fmla}
A_\lambda(1)=\frac1\rho\sum_{d\mid m}\frac{\mu(d)}{d^2}
\frac1{P_1(m/d)P_N(m/d)}\prod_{i\geq  1}\binom{\rho P_i(m/d)-1+m_i/d}{m_i/d}
\end{equation}
where $\lambda=(1^{m_1}2^{m_2}\cdots N^{m_N})$ with $N=l(\lambda)$,
the length of $\lambda$, 
$$
P_i(m):=\sum_{j\geq 1}\min(i,j)\,m_j, \qquad m:=(m_1,m_2,\ldots), 
\quad \rho=:2g-2,
$$
and $\mu$ is the M\"obius function of number theory.
\end{theorem}
Note that if the entries of $m$ have no
common factor then the right hand side of~\eqref{main-fmla} consists
of only one term.

\begin{corollary}
  As a function of $g$, the quantity $A_\lambda(1)$ is a polynomial of
  degree $l(\lambda)-1$; its leading coefficient in $\rho:=2g-2$ is
\begin{equation}
\label{leading-term}
\frac 1{P_1(m)P_{l(\lambda)}(m)}\prod_{i\geq 1} \frac{P_i(m)^{m_i}}{m_i!}.
\end{equation}
\end{corollary}
\begin{remark}
  In particular, we recover the fact (noticed numerically in
  \cite{HRV} and proved in greater generality in \cite{HeRV}) that
  $A_n(1)$ is a polynomial in $\rho$ of degree $n-1$ and leading
  coefficient $n^{n-2}/n!$. The appearance of the term $n^{n-2}$, the
  number of spanning trees on $n$ labelled vertices, is not a
  coincidence, see Remark~\ref{trees}.
\end{remark}

We also note the following important property (here we write
$A_\lambda^\rho$ with $\rho=2g-2$ for $A_\lambda$ to indicate the
dependence on $g$), which was inspired by the interpretation of
$A_\lambda(1)$ in terms of the moduli space of Higgs bundles mentioned
above.
\begin{proposition}
\label{scaling}
  Let $n$ be a positive integer and
  $\lambda=(\lambda_1,\lambda_2,\ldots)$ a non-zero partition. Define
  $n\lambda:=(n\lambda_1,n\lambda_2,\ldots)$. Then
$$
A_{n\lambda}^\rho(q)=A_\lambda^{n\rho}(q), \qquad \qquad \rho:=2g-2.
$$
In particular,
$$
A_{(n)}(q)=q^{n(g-1)+1}
$$
\end{proposition}

\paragraph{Acknowledgements} The work was started while the author was
visiting Oxford University supported by an FRA from the University
of Texas at Austin, EPSRC grant EP/G027110/1 and Visiting Fellowships
at All Souls and Wadham Colleges in Oxford.  The note was finished
while the author attended the workshop {\it Representation Theory of
  Quivers and Finite Dimensional Algebras} in February 2011 at MFO,
Oberwolfach. He would like to thank the organizers for the invitation
to participate as well as MFO for the hospitality. He would also like
to thank the participants of the workshop T. Hausel, E. Letellier,
S. Mozgovoy and M. Reineke for many stimulating conversations. The
author's research is supported by the NSF grant DMS-0200605 and a
Research Scholarship from the Clay Mathematical Institute.

\section{Definition of $A_\lambda(q)$}
Hua's formula expresses the $A$-polynomial of a quiver in terms of
generating series.  Let
\begin{equation}
\label{Hua-defn}
\Hua(T;q):=\sum_{\lambda} \frac{q^{(g-1)\langle \lambda, \lambda
    \rangle}}{b_\lambda(q^{-1})} \, T^{|\lambda|}.
\end{equation}
Then for $S_g$ we have
\begin{equation}
\label{Hua-fmla}
(q-1)\,\Log\left(\Hua(T;q)\right)=\sum_{n\geq 0} A_n(q)T^n.
\end{equation}

Here $\Log$ is the plethystic equivalent of the usual $\log$ (see for
example \cite{HRV} for a discussion). The main use of $\Log$ is as a
convenient tool to manipulate the conversion of series to infinite
products. It takes a factor of the form $(1-w)^{-1}$, where $w$ is a
monomial in some set of variables, to $w$.

We are interested in the values $A_n(1)$. Because the individual
terms in the generating function $\Hua(T;q)$ have high order poles at
$q=1$ it is not easy to recover these numbers from
\eqref{Hua-fmla}.

Consider Hua's formula as the limit as $N\rightarrow \infty$ of the
truncated series
\begin{equation}
\label{Hua_N-defn}
\Hua_N(T;q):=\sum_{\lambda_1\leq N} \frac{q^{(g-1)\langle \lambda, \lambda
    \rangle}}{b_\lambda(q^{-1})} \, T^{|\lambda|}.
\end{equation}
Using multiplicities $m=(m_1,\ldots,m_N)$ to parametrize partitions
we can write this series as the specialization $x_i=T^i$ for
$i=1,2,\ldots,N$ of the following series in the variables
$x:=(x_1,x_2,\ldots,x_N)$ and
$$
\Hua_N(x;q)
:=\sum_m \frac{q^{(g-1){}^tm\calH_Nm}} 
{(q^{-1})_m}\,x^m
$$
where 
$$
x^m:=x_1^{m_1}x_2^{m_2}\cdots x_N^{m_N}\qquad
(q)_m:= \prod_{i=1}^N(q)_{m_i}\qquad
\calH_N:=\left(\min(i,j)\right),\qquad i,j=1,2,\ldots,N.
$$
 It is remarkable that the
inverse of $\calH_N$ has a very simple and sparse structure, namely,
$$
\calH_N^{-1}=
\left(
\begin{matrix}
2&-1 &0 &&\\
-1& 2&-1&&\\
0&-1&2&-1 &\\
&&\vdots& &\\
& &-1&2&-1\\
& &0&-1&1
\end{matrix}
\right).
$$
This is the Cartan matrix of the {\it tadpole} $T_N$ (obtained by
folding the $A_{2N}$ diagram in the middle), a positive definite
symmetric matrix of determinant~$1$.

We define refinements of the $A$-polynomial by replacing $\Hua(T;q)$
by $\Hua_N(x;q)$. More precisely, define $A_\lambda(q)$ for $\lambda$
a partition with $l(\lambda)\leq N$ as $A_m(q)$, where 
\begin{equation}
\label{A-refined-defn}
(q-1)\,\Log\left(\Hua_N(x;q)\right)=\sum_m A_m(q)\,x^m,
\end{equation}
and $m=(m_1,\ldots,m_N)$ are the multiplicities of $\lambda$, so
$\lambda=(1^{m_1}2^{m_2}\cdots N^{m_N})$. It is straightforward to
check that the definition is independent of $N$ as long as
$l(\lambda)\leq N$

As mentioned, $A_\lambda(q)$ is a priori a rational function of $q$
but we expect it to actually be a polynomial.
We have in  any case
$$
A_n=\sum_{|\lambda|=n} A_\lambda.
$$

\begin{remark}
\label{trees}
Appropriately scaled, $\Hua_N(x;q)$ converges as $g\rightarrow \infty$
to
$$
G(x;q):=\sum_mq^{{}^tm\calH_Nm}\,\frac{x^m}{m!}, \qquad \qquad m!:=m_1!\cdots
m_N!.
$$
By the exponential formula of combinatorics 
$$
\left.(q-1)\log G(x;q)\right|_{q=1}
$$
is the exponential generating function for certain weighted
trees. This gives an interpretation of the leading term of
$A_\lambda(1)$ as $g\rightarrow \infty$ in terms of weighted trees in
the spirit of~\cite{HeRV}. See Remark~\ref{polya-remark}.
\end{remark}

\begin{proof}[Proof of Proposition~\ref{scaling}]
Let $m=(m_1,\ldots,m_N)$ be the multiplicities of $\lambda$. Then the
multiplicities of $n\lambda$ are
$$
m[n]:=(0,\ldots,m_1,0,\dots,m_2,0,\ldots,m_N),
$$
where $m_i$ is located at the spot $n i$. Note that $m$ and $m[n]$
have exactly the same non-zero entries. Hence $(q)_{m[n]}=(q)_m$. Also,
  it is easy to check that
$$
{}^tm[n]\calH_{nN}m[n]=n\,{}^tm\calH_Nm.
$$
Now the first claim follows from the definition of $A_\lambda$.  
The second claim follows from the first since $A_1(q)=q^g$.
\end{proof}

\section{Truncation to order $N=1$}
\label{order-1}

We will first consider the case $N=1$ in detail as it contains  all
of the main ingredients of the general case.  When $N=1$ we can
interpret $\Hua_N$ as follows. Consider the action by conjugation of
$\GL_n(\F_q)$ on $g$-tuples of $n\times n$ matrices
$X:=(X_1,\ldots,X_n)$ with coefficients in $\F_q$. The number $N_n(q)$
of such $g$-tuples, each weighed by $1/|\Stab(X)|$, where $\Stab(X)$
is the stabilizer of $X$ in $\GL_n(\F_q)$, is
$$
N_n(q)=\frac{q^{gn^2}}{|\GL_n(\F_q)|}=\frac{q^{(g-1)n^2}}{(q^{-1})_n},
$$
since
$$
|\GL_n(\F_q)|=(-1)^nq^{\tfrac12n(n-1)}(q)_n=q^{n^2}(q^{-1})_n
$$
Hence
$$
\Hua_1(x_1;q)=\sum_{n\geq 0} N_n(q)\,x_1^n.
$$

We want to get an expression for $A_{(1^n)}(1)$ and we do this by
studying the asymptotics of $\log \Hua(x;q)$ as $q\rightarrow 1$.  The
general question of describing the asymptotics of series like \eqref
{Hua-defn} as $q$ approaches $1$ has a long history. Already Ramanujan
used such an asymptotics to test his famous $q$-series identities now
known as the Rogers--Ramanujan formulas. More recently, the question
arose in conformal field theory and there is a beautiful conjecture of
Nahm~\cite{Nahm} that relates the modular behaviour of $q$-series of
this type with torsion elements in the Bloch group.

To conform with the standard format in the literature we will change
$q$ to $q^{-1}$.  The basic result is the following
(see~\cite{Mcintosh},\cite{Zagier}).
\begin{proposition}
\label{asymptotics}
  Let $a$ and $T$ be  fixed positive real numbers and let $z$ be the
  positive real root of
\begin{equation}
\label{trinomial-eqn}
Tz^a+z-1=0.
\end{equation}
Then with $q=e^{-t}$ and $t\searrow 0$ we have
$$
\log \sum_{n \geq 0}\frac{q^{\tfrac12
    an^2}}{(q)_n}\,T^n=c_{-1}t^{-1}+c_0+c_1t+\cdots,
$$
where
$$
c_{-1}=\Li_2(1-z)+\tfrac12 a\log^2 z, \qquad \qquad 
c_0=-\tfrac12 \log(z+a(1-z))
$$

\end{proposition}

It is a classical fact going back to Lambert that the trinomial
equation~\eqref{trinomial-eqn} can be solved in terms of
hypergeometric functions. Concretely, we can expand $z$ as a power
series in $T$ and obtain
\begin{equation}
\label{landen}
z=z(T)=\sum_{n\geq 0} \frac1{(a-1)n+1}\binom{an}n\,(-T)^n.
\end{equation}
This result follows easily from Lagrange's formula~\cite[p. 133]{WW}
(see also~\cite[p. 43]{Polya}).
For example, if $a=2$ then
$$
z=\frac{-1+\sqrt{1+4T}}{2T} = 1 - T + 2T^2 - 5T^3 + 14T^4 - 42T^5 +
132T^6  +\cdots,
$$
where the coefficients are, up to a sign, the Catalan numbers. 

We would now like to combine~\eqref{landen} with
Proposition~\ref{asymptotics}. To find an expression for $c_{-1}$ as a
power series in $T$ we differentiate~\eqref{trinomial-eqn} with
respect to $T$ to find
$$
\frac{\partial
  z}{\partial T}=-\frac{z^a}{1+az^{a-1}T}=-\frac{(1-z)z}{T(z+a(1-z))}.
$$
On the other hand
$$
\frac{\partial
  c_{-1}}{\partial T}=\log z\left(\frac1{1-z}+\frac az\right)\frac{\partial
  z}{\partial T}
$$
and these combine to give
$$
T\frac{\partial
c_{-1}}{\partial T}=-\log z
$$
As it happens, $\log z$ also has an explicit hypergeometric power
series expansion~\cite[p. 134, Ex. 3]{WW}. In fact, all powers $z^s$
have such an expansion, namely
$$
z^s=1+s\sum_{n\geq 1} \binom{s-1+an}n\frac{(-T)^n}{(a-1)n+s},
$$
from which we obtain by differentiation with respect to $s$ and
setting $s=0$
$$
\log z= \sum_{n\geq 1}\frac1a\binom{an}n\,\frac{(-T)^n}n.
$$ 
It follows that
$$
c_{-1}=\sum_{n\geq 1}(-1)^{n-1}\frac1a\binom{an}n\,\frac{(-T)^n}{n^2}.
$$
and therefore
$$
A_{(1^n)}(1)=\frac1{n^2}\sum_{d\mid n}\mu\left(\frac
  nd\right)\frac1\rho \binom{\rho d +d-1}d, \qquad \qquad \rho=2g-2.
$$
\begin{remark}
\label{polya-remark}
  As observed by Polya~\cite[p. 44]{Polya} if we set $z=1+w/a,T=-U/a$
  and let $a$ go to infinity then~\eqref{trinomial-eqn} becomes
$$
Ue^w=w
$$
and $z(T)$ becomes $w(U)=\sum_{n\geq1} n^{n-1}\,U^n/n!$, the
exponential generating function for labelled rooted trees.
\end{remark}

\section{General case}
\label{general}
\subsection{Lagrange's inversion formula}
Consider
$$
\left.(q-1)\log\left[\Hua_N(x;q)\right]\right|_{q=1}
$$
The asymptotic expansion now involves $N$ saddle points
$z_1,\ldots,z_N$, which are solutions to the system of equations
\begin{equation}
\label{saddle-pts-eqn}
1-z_i=x_i\prod_{j=1}^Nz_j^{a_{i,j}}, \qquad \qquad i=1,2,\ldots, N,
\end{equation}
where $a_{i,j}=-\rho \min(i,j)$. These equations determine
$z_i=z_i(x)$ implicitly.  

We will carry on the discussion of solving~\eqref{saddle-pts-eqn} for
an generic symmetric matrix $A=(a_{i,j})\in \Z^{N\times N}$ as far as
we can before specializing to our situation. As in the one variable
case of the previous section we can expand $z_i$ as power series in
the $x_j$'s. We obtain expressions for the corresponding coefficients
by applying a multi-variable version of Lagrange's formula due to
Stieltjes (see~\cite{debruijn} and~\cite{sack} for other
multi-variable versions and some history on the matter). Note that
$z_i(0)=1$.

\begin{remark}
  The kind of analysis done in this section appears prominently in the
  physics literature under the heading of $Q$-systems, originating
  from the work of Kirillov--Reshetikhin on representation theory and
  the combinatorics of the Bethe Ansatz. There is a substantial
  literature on the subject. The basic application of Lagrange's
  inversion can be found for example in~\cite{Kuniba-et-al}; see also
  Nahm's paper~\cite{Nahm} already mentioned. We preferred to rederive
  the results we needed from scratch.
\end{remark}
\begin{theorem}
\label{lagrange}
Let $z_1,\ldots, z_N$ be implicitly given by
\begin{equation}
\label{z-implicit}
z_i=y_i+x_i\,f_i(z_1,\ldots,z_N), \qquad \qquad i=1,2,\ldots,N,
\end{equation}
where $f_1,\ldots,f_N$ are analytic. Then for $g$ analytic we have
$$
g(\z)=\frac1{D}\sum_{m}a_m(y)\,
\frac{x^m} {m!},
$$
where $m=(m_1,\ldots,m_N),x=(x_1,\ldots,x_N)$, etc.,
$x^m:=x_1^{m_1}\cdots x_N^{m_N}$, $m!:=m_1!\cdots m_N!$,
$$
a_m(y):=
\frac{\partial^{m_1}}{\partial y_1^{m_1}}
\cdots \frac{\partial^{m_N}}{\partial y_N^{m_N}}
\left[g(y)f_1(y)^{m_1}\cdots 
f_N(y)^{m_N} \right],
$$
and
$$
D:=\det\left(\frac{\partial z_i}{\partial y_j}\right)
$$
\end{theorem}
By differentiating~\eqref{z-implicit} with respect to the $y$'s we see
that 
\begin{equation}
\label{partial-z-fmla}
\left(I_N-\left(x_i\frac{\partial f_i}{\partial z_j}\right)\right) 
\left(\frac{\partial z_i}{\partial y_j} \right)= I_N,
\end{equation}
where $I_N$ is the $N\times N$ identity matrix and hence
\begin{equation}
\label{delta-fmla-1}
D^{-1}=\det \left(I_N-\left(x_i\frac{\partial f_i}{\partial
      z_j}\right)\right) 
\end{equation}

We apply this theorem to $g(z_1,\ldots,z_N)=z_1^{s_1}\cdots z_N^{s_N}$
and $f_i(z_1,\ldots,z_N)=-\prod_{j=1}^Nz_j^{a_{i,j}}$. Note that the
Jacobian matrix $\partial z_i/\partial x_j$ at $x=0$ is the identity
matrix. Setting $y_i=1$ we obtain
\begin{equation}
\label{z-monom-expansion}
  z_1^{s_1}\cdots z_N^{s_N}=\frac1D \sum_m
  (-1)^{m_1+\cdots m_N}\prod_{i=1}^N
  \binom{s_i+P_i(m)}{m_i}\,x^m,
\end{equation}
where $P_i(m):=\sum_{j=1}^N a_{j,i}\,m_j$ and $z_1,\ldots,z_N$
satisfy~\eqref{saddle-pts-eqn}.  In particular, taking
$s_1=\cdots=s_N=0$ we find
\begin{equation}
\label{delta-fmla}
D= \sum_m
  (-1)^{m_1+\cdots m_N}\prod_{i=1}^N
  \binom{P_i(m)}{m_i}\,x^m.
\end{equation}
\begin{remark}
  Note that~\eqref{delta-fmla} implies that the power series on the
  right hand side is an algebraic function of $x_1,\ldots,x_N$ for any
  choice $a_{i,j}\in \Q$. This generalizes an observation of Hurwitz
  in the one variable case (see~\cite[footnote 1, p. 43]{Polya}).
\end{remark}

By~\eqref{delta-fmla-1} 
\begin{equation}
\label{delta-fmla-2}
D^{-1}=\det\left(I_N-FAZ^{-1}\right),
\end{equation}
where $F$ and $Z$ are the diagonal matrices
$$
F:=\diag(x_1f_1,\ldots,x_Nf_N), \qquad \qquad
Z:=\diag(z_1,\ldots,z_N).
$$
In particular, $D^{-1}$ is a Laurent polynomial in the $z_j's$.

For example, in the dimension $N=1$ case of~\S\ref{order-1}, with
$z=z(T)=1+O(T)$ a solution to $1-z=Tz^a$ for a generic $a\in \Z$, we find
$D^{-1}=1+xaz^{a-1}$ and hence
$$
D=\frac{z}{1+a(1-z)}=\sum_{n\geq 0}\binom{an}n\,(-T)^n.
$$
\begin{lemma}
\label{minors}
Let $A=(a_{i,j})$ be an $N\times N$ matrix and $x_1,\ldots,
x_N$ independent variables. Then
$$
\det\left(\diag(x_1,\ldots,x_N)-A\right)=\sum_I (-1)^{N-\#
  I}D_I\,x^I,
$$
where $I$ runs over all subsets of $\{1,2,\ldots,N\}$, $x^
I:=\prod_{i\in I} x_i$ and $D_I$ is the  
determinant of the matrix obtained from $A$ by striking the rows and
columns whose indices are in $I$.
\end{lemma}
Combining the lemma with~\eqref{delta-fmla-2} we get
\begin{equation}
\label{delta-fmla-3}
D^{-1}=\sum_I (-1)^{\#I}D_{\bar I}\,\prod_{i\in I} \frac{x_if_i}{z_i},
\end{equation}
where $\bar I$ is the complement of $I$.

We would like an explicit form for the series expansion of
$z_1^{s_1}\cdots z_N^{s_N}$.  To do this
combine~\eqref{z-monom-expansion}, which already yields
$z_1^{s_1}\cdots z_N^{s_N}$ as a ratio of two power series,
with~\eqref{delta-fmla-3} to obtain
\begin{eqnarray*}
  z_1^{s_1}\cdots z_N^{s_N}
&=&
\sum_I D_{\bar I} \sum_m
  (-1)^{m_1+\cdots m_N}
\prod_{j=1}^N
  \binom{s_j+\sum_{k=1}^N a_{k,j}\,m_k+\sum_{i\in
      I}(a_{i,j}-\delta_{i,j})} {m_j}\,x^Ix^m\\
 &=&  
\sum_m  (-1)^{m_1+\cdots m_N} \sum_I (-1)^{\# I} D_{\bar I} 
\prod_{i=1}^N
  \binom{s_i+P_i(m)-\delta_{I,i}}
  {m_i-\delta_{I,i}}\,x^m\\
&=&
\sum_m  (-1)^{m_1+\cdots m_N} 
\prod_{i=1}^N
  \binom{s_i+P_i(m)} {m_i}\, R(m;s)
\,x^m,
\end{eqnarray*}
where $\delta_{I,i}:=1$ if $i\in I$ and is $0$ otherwise and
$$
R(m;s):=\sum_I 
(-1)^{\# I} D_{\bar I} \prod_{i\in I}\frac{m_i}{s_i+P_i(m)}.
$$
Note that it follows from the above identity that $R(m;0)=0$ for $m$
non-zero. We can write $R$ back as a determinant using
Lemma~\ref{minors}.  By the generalized matrix-tree theorem we can
interpret this determinant in terms of weighted spanning trees on $N$
labelled vertices. For a generic symmetric matrix $A$ we have
\begin{equation}
\label{partR}
\left.\frac{\partial}{\partial s_i}R(m;s)\right|_{s=0} =
\frac{m_i\tau(m)}{\prod_{i=1}^NP_i(m)},
\end{equation}
where 
$$
\tau(m):=\sum_{T}\prod_{i=1}^N m_i^{d_i-1}\prod_e a_{i,j}.
$$
Here the second product runs over all edges
$i\stackrel{e}{\rightarrow} j$ of the tree $T$ and $d_i$ is the number
of edges connecting to the vertex $i$ in $T$.

It turns out that $\tau(m)$ has a particularly simple form when
$A=\calH_N$; this is not true in  general.
\begin{lemma}
\label{tau}
For $A=\calH_N$ we have
$$
\tau(m)=P_2(m)\cdots P_{N-1}(m).
$$
\end{lemma}

\subsection{Schl\"afli's differential}
For a generic symmetric matrix $A=(a_{i,j})\in \Z^{N\times N}$ define
$$
V(x)= \lim_{q\rightarrow 1}\left\{(q-1)\log\left[\sum_m
\frac{q^{\tfrac12{}^tmAm}}{(q)_m}x^m
\right]\right\}
$$
and let $z_i=z_i(x)$ be the power series solutions to the saddle point
equations~\eqref{saddle-pts-eqn} of the previous section. Then $V$ and
$z_i$ are related as follows.
\begin{proposition}
\label{schlaefli}
We have
$$
dV(x)=\sum_{i=1}^N \log z_i(x) \,\frac{dx_i}{x_i}.
$$
\end{proposition}

\begin{remark}
  We call $dV$ Schl\"afli's differential because of its analogy to the
  situation in hyperbolic geometry (see~\cite{Milnor}). There $V$ is
  the volume of a polyhedron in hyperbolic $3$-space, $\log z_i$
  correspond to the length of an edge and $dx_i/x_i$ to $d\theta_i$,
  where $\theta_i$ is the dihedral angle of the polytope at that
  edge. Although we will not emphasize this aspect, $V$ can be
  expressed as a sum of dilogarithms of the $z_i$'s. This is what we
  see in Proposition~\ref{asymptotics}.
\end{remark}

\subsection{Proof of Theorem~\ref{main}}
We return to the case of interest, where
$A=-\rho\,\calH_N$.  It follows from Proposition~\ref{schlaefli} that 
\begin{equation}
\label{z-prod}
z_i=\prod_m(1-x^m)^{m_i\,A_m(1)}
\end{equation}

Given a non-zero $m$ chose $i$ such that $m_i\neq
0$. Formula~\eqref{main-fmla} now follows from~\eqref{partR} and
Lemma~\ref{tau} by taking the logarithm of~\eqref{z-prod}. The
M\"obius inversion in~\eqref{main-fmla} accounts for the difference
between $\log$ and $\Log$.

\end{document}